\documentclass[11pt]{amsart}
\usepackage{amsmath,amsthm,amssymb}
\usepackage{enumerate}

\usepackage[colorlinks=true, urlcolor=black, citecolor=black, linkcolor=black, hyperfootnotes=true]{hyperref}

\usepackage{url}

\usepackage[english]{babel}
\usepackage[utf8]{inputenc}
\usepackage[a4paper]{geometry}
\usepackage{latexsym}
\usepackage{amscd}
\usepackage{graphics}
\usepackage{color}
\usepackage{array}
\usepackage{mathrsfs}
\usepackage{graphicx}
\usepackage{stmaryrd}
\usepackage{mathabx}
\usepackage{dsfont}
\usepackage{comment}

\newtheorem{thm}{Theorem}[section]

\newtheorem{lem}[thm]{Lemma}

\newtheorem{prop}[thm]{Proposition}

\newcounter{claimcounter}[thm]
\numberwithin{claimcounter}{thm}
\newtheorem{claim}[claimcounter]{Claim}

\theoremstyle{definition}
\newtheorem{defin}[thm]{Definition}

\numberwithin{equation}{section}

\newcommand{\N}{\mathbb{N}}
\newcommand{\R}{\mathbb{R}}
\renewcommand{\omega}{\N}

\DeclareMathOperator{\supp}{supp}

\DeclareMathOperator{\CS}{CS}
\DeclareMathOperator{\osc}{osc}
\DeclareMathOperator{\diam}{diam}

\makeatletter
\newcommand{\alaligne}{~\vspace*{\topsep}\nobreak\@afterheading}
\makeatother

\begin{document}

\title{Oscillation stability by the Carlson--Simpson theorem}

\author[T. Bice]{Tristan Bice}
\author[N. de Rancourt]{No\'e de Rancourt}
\author[J. Hubi\v{c}ka]{Jan Hubi\v{c}ka}
\author[M. Kone\v{c}n\'y]{Mat\v{e}j Kone\v{c}n\'{y}}

\address[T. Bice]{Institute of Mathematics of the Czech Academy of Sciences, \v{Z}itn\'a 609/25,
11000 Praha 1
Czech Republic}
\email{bice@math.cas.cz}

\address[N.~de Rancourt]{Universit\'e de Lille, CNRS, UMR 8524 – Laboratoire Paul Painlev\'e, F-59000 Lille, France}
\email{nderancour@univ-lille.fr}

\address[J. Hubi\v{c}ka]{Department of Applied Mathematics (KAM), Charles University, Malostransk\'e n\'am\v{e}st\'\i{}
 25, Praha 1, Czech Republic}
\email{hubicka@kam.mff.cuni.cz}

\address[M. Kone\v{c}n\'y]{Institute of Algebra, TU Dresden, Helmholtzstra{\ss}e 10, 01069 Dresden, Germany \and Department of Applied Mathematics (KAM), Charles University, Malostransk\'e n\'am\v{e}st\'\i{}
 25, Praha 1, Czech Republic}
\email{matej.konecny@tu-dresden.de}



\thanks{N. de Rancourt acknowledges support from the Labex CEMPI (ANR-11-LABX-0007-01).}

\thanks{Research of Tristan Bice was supported by GA\v{C}R project 22-07833K and RVO: 67985840.}
\thanks{In the earlier stages of this project, research of Jan Hubi\v{c}ka and Mat\v{e}j Kone\v{c}n\'{y} was supported by the project 21-10775S of
the Czech Science Foundation (GA\v{C}R). In the later stages, Jan Hubi\v{c}ka was supported by a project that
has received funding from the European Research Council under the European Union’s Horizon
2020 research and innovation programme (grant agreement No 810115), and Mat\v{e}j Kone\v{c}n\'{y} was supported
by a project that has received funding from the European Union (Project POCOCOP, ERC
Synergy Grant 101071674). Views and opinions expressed are however those of the authors only
and do not necessarily reflect those of the European Union or the European Research Council
Executive Agency. Neither the European Union nor the granting authority can be held responsible
for them.}

\begin{abstract}
We prove oscillation stability for the Banach space $\ell_\infty$: every weak-* Borel, uniformily continuous map from the unit sphere of this space to a compact metric space can be made arbitrarily close to a constant map when restricted to the unit sphere of a suitable linear isometric subcopy of $\ell_\infty$. We also give a new proof of oscillation stability for the Urysohn sphere (a result by Nguyen Van Th\'e--Sauer): every uniformily continuous map from the Urysohn sphere to a compact metric space can be made arbitrarily close to a constant map when restricted to a suitable isometric subcopy of the Urysohn sphere. Both proofs are based on Carlson--Simpson's dual Ramsey theorem.
\end{abstract}

\maketitle

\section{Introduction}

The \textit{oscillation} of a map $f \colon X \to M$, where $M$ is a metric space, is defined as $\osc(f) = \sup_{x, y \in X} d(f(x), f(y))$. If $X$ is a Banach space, denote by $S_X$ its unit sphere, i.e. the set of its norm $1$ vectors. 
Say that a map $f \colon S_X \to M$, where $M$ is a metric space, \textit{stabilizes}, if for every $\varepsilon > 0$, there exists a vector subspace $Y \subseteq X$ linearly isometric to $X$ such that $\osc(f\restriction_{S_Y}) \leqslant \varepsilon$. It is classical result by James \cite{JamesDistortion} that when $X$ is the Banach space $c_0$ or $\ell_1$, every equivalent norm on $X$ stabilizes. On the other hand, Odell--Schlumprecht \cite{OdellSchlumprecht} solved in 1994 the longstanding \textit{distortion problem} by showing that for each $p \in (1, \infty)$, there is an equivalent norm on the Banach space $\ell_p$ that does not stabilize ; we say that $\ell_p$ is \textit{distortable}.

\smallskip

Closely related to the question of distortion is this of \textit{oscillation stability}, concerned with the stabilization of not only norms, but more generally arbitrary uniformily continuous functions. It follows from Odell--Schlumprecht's solution of the distortion problem that for each $p \in [1, \infty)$, there exists a Lipschitz function $f \colon S_{\ell_p} \to [0, 1]$ that does not stabilize (we refer the reader to Odell--Schlumprecht's survey \cite{OdellSchlumprechtSurvey} for more details on the distortion problem and its links with oscillation stability). On the opposite, it was shown by Gowers \cite{GowersLipschitz} that every uniformily continuous map $f \colon S_{c_0} \to K$, where $K$ is a compact metric space, stabilizes. It is worth noticing that the latter result has a Ramsey-theoretic flavour: the map $f$ can be seen as a colouring of $S_{c_0}$ with compactly many colours, and the result asserts the existence of a nearly monochromatic subcopy of $S_{c_0}$. This remark is not anecdotal, as the proof of Gowers' theorem is entirely combinatorial.

\smallskip

The first aim of this paper is to prove the following oscillation stability result for the Banach space $\ell_\infty$.

\begin{thm}\label{thm:OscStabEllInf}
Every weak-* Borel, uniformily continuous map $f \colon S_{\ell_\infty} \to K$, where $K$ is a compact metric space, stabilizes.
\end{thm}

\noindent As for Gowers' $c_0$ theorem, our proof of Theorem \ref{thm:OscStabEllInf} is combinatorial; namely, it is based on Carlson--Simpson's dual Ramsey theorem \cite{CarlsonSimpson}. In the statement of Theorem \ref{thm:OscStabEllInf}, the word \textit{weak-*} refers to the weak-* topology one gets when seeing $\ell_\infty$ as the dual of $\ell_1$; it coincides with the topology on $S_{\ell_\infty}$ induced by the product topology on $\R^\N$. The assumption that $f$ is weak-* Borel cannot be removed from the statement of Theorem \ref{thm:OscStabEllInf}, as will be shown in Proposition \ref{prop:diagonal} (it can be weakened to, e.g., $f$ being weak-* Souslin-measurable, but we decided to restrict our statement to Borel measurability in order to keep this paper accessible to a large audience). It is a definability assumption, made in order to avoid pathological functions $f$, constructed for instance with the help of the axiom of choice. The requirement of such assumptions is classical in infinite-dimensional Ramsey theory; see \cite{TodorcevicOrange} for more details. Note that, unlike for $c_0$, $f$ being uniformily continuous cannot be considered as a definability assumption since $\ell_\infty$ is nonseparable.

\smallskip

Towards a better understanding of the combinatorial structure of the Hilbert space $\ell_2$, and in the hope for a purely combinatorial solution to the distortion problem (Odell--Schlumprecht's argument relying on advanced Banach space theoretic techniques), oscillation stability has also been studied in the case of the \textit{Urysohn sphere} $\mathcal{S}$, a metric space that can be seen as a good combinatorial analogue of the unit sphere of $\ell_2$. The Urysohn sphere having no linear structure, definitions need to be adapted. Say that a map $f\colon X \to M$ between two metric spaces \textit{stabilizes} if for every $\varepsilon > 0$, there exists a subset $Y \subseteq X$ isometric to $X$ such that $\osc(f\restriction_Y) \leqslant \varepsilon$. The following oscillation-stability result has been proved by Nguyen Van Thé--Sauer \cite{NVTSauer}, building on earlier results by Lopez-Abad--Nguyen Van Thé \cite{LopezAbadNVT}.

\begin{thm}[Nguyen Van Thé--Sauer 2009]\label{thm:NVTS}
    Every uniformily continuous map $f \colon \mathcal{S} \to K$, where $K$ is a compact metric space, stabilizes.
\end{thm}

\noindent Nguyen Van Thé--Sauer's original proof is combinatorial and done by hand; our second aim in this paper is to give a shorter proof of this result, also based on the Carlson--Simpson theorem.

\smallskip

The present paper is part of a wider ongoing project aiming at defining and understanding a metric version of the notion of \textit{big Ramsey degrees}. The theory of big Ramsey degrees studies analogues of the infinite Ramsey theorem in sets endowed with structure, such as ordered sets, graphs or discrete metric spaces, and our general goal is to make this theory compatible with the study of continuous metric structures such as metric spaces or Banach spaces. Theorems \ref{thm:OscStabEllInf} and \ref{thm:NVTS} are particular cases of metric big Ramsey degree results for the Banach space $\ell_\infty$ and the Urysohn sphere, to appear in a forthcoming paper by the same authors; part of those results have been announced in \cite{MetricBRDEurocomb}, where our setting for metric big Ramsey degrees is also presented.  The proof methods in use in the present paper are heavily inspired by recent methods for computing big Ramsey degrees based on the Carlson--Simpson theorem, first developed by Hubi\v{c}ka \cite{HubickaCarlsonSimpson} and then expanded by Balko--Chodounsk\'y--Hubi\v{c}ka--Kone\v{c}n\'y--Ne\v{s}et\v{r}il--Vena \cite{BRDForbiddingCycles} and Balko--Chodounsk\'y--Dobrinen--Hubi\v{c}ka--Kone\v{c}n\'y--Vena--Zucker \cite{BRDPartialOrder}.

\smallskip

The present paper is organized as follows. In Section \ref{sec:CarlsonSimpson}, we recall and explain the Carlson--Simpson theorem. In Section \ref{sec:EllInf}, we prove Theorem \ref{thm:OscStabEllInf} and show that the \mbox{weak-*} Borelness assumption cannot be removed. Finally, in Section \ref{sec:Urysohn}, we present our proof of Theorem \ref{thm:NVTS}.

\bigskip

\section{The Carlson--Simpson theorem}\label{sec:CarlsonSimpson}

In this paper, $\N$ will denote the set $\{0, 1, 2, \ldots\}$. Given an integer $k \in \N$, we will follow the usual set-theoretic convention and identify $k$ to the set $\{0, \ldots, k-1\}$ of its predecessors.

\smallskip

\textit{Ramsey theory} is a collection of rather heterogeneous results sharing a similar spirit: for every reasonable colouring of a sufficiently large structure, one should be able to find a large monochromatic substructure. For instance, the \textit{infinite Ramsey theorem}, which gave its name to the theory, asserts that for every integer $d \geqslant 1$ and every colouring of $[\N]^d$ with a finite number of colours, there should exist an infinite subset $M \subseteq \N$ such that $[M]^d$ is monochromatic (here, $[M]^d$ denotes the set of all $d$-elements subsets of the set $M$). 

\smallskip

While in the infinite Ramsey theorem, one colours subsets, or equivalently injections, the Carlson--Simpson theorem deals with colourings of partitions, or equivalently surjections. This is why it is often refered to as a \textit{dual Ramsey theorem}. Denote by $(\N)^\infty$ the set of all partitions of $\N$ into infinitely many pieces (the pieces are required to be nonempty). Given $P \in (\N)^\infty$, denote by $(P)^\infty$ the set of all $Q \in (\N)^\infty$ that are \textit{coarser than $P$}, i.e. that can be obtained by merging together pieces of $P$. The Carlson--Simpson's theorem essentially says that for every definable enough colouring of $(\N)^\infty$ into finitely many colours, there exists $P \in (\N)^\infty$ such that $(P)^\infty$ is monochromatic. In order to give a more formal statement of this result and in particular of its definability assumption, we will first give a canonical representation of elements of $(\N)^\infty$ as surjections. Recall that if $L$ is a linearly ordered set, an \textit{initial segment} of $L$ is any subset $A \subseteq L$ satisfying the following property: for every $x \in A$ and every $y \in L$ such that $y \leqslant x$, one should have $y \in A$.

\begin{defin}
    Let $K, L$ be linearly ordered sets. A surjection $L \to K$ is said to be \textit{rigid} if for every initial segment $A$ of $L$, $f(A)$ is an initial segment of $K$.
\end{defin}

\noindent In practice, the only linearly ordered sets we will consider in the present paper will be either finite or equal to $\N$.  In those cases, it is easy to see that a surjection $f \colon K \to L$ is rigid iff for every $x, y \in L$ with $x < y$, one has $\min f^{-1}(\{x\}) < \min f^{-1}(\{y\})$.

\smallskip

Denote by $\CS$ (like \textit{Carlson--Simpson}) the set of all rigid surjections $\N \to \N$. This is a monoid for composition. We endow it with the topology induced by the product topology on $\N^\N$; it follows from standard results that with this topology, $\CS$ is a Polish space (although this fact won't be needed in this paper). To each $f \in \CS$, we can associate the partition $\{f^{-1}(\{n\}) \mid n \in \N\} \in (\N)^\infty$; it is easy to see that this correspondence is a bijection between $\CS$ and $(\N)^\infty$. Moreover, given $p, q \in \CS$ and their respective associated partitions $P$ and $Q$, it is easy to check that $Q$ is coarser than $P$ iff there exists $r \in \CS$ with $q = r \circ p$. Hence the Carlson--Simpson theorem can be phrased as follows.

\begin{thm}[Carlson--Simpson 1984]
    For every $k \in \N$ and every Borel colouring $c \colon \CS \to k$, there exists $p \in \CS$ such that $c$ is constant on $\CS \circ p$.
\end{thm}

This theorem has first been proved in \cite{CarlsonSimpson}, where it has also be shown that the assumption that $c$ is Borel cannot be removed (although it can be relaxed to, e.g., $c$ being Souslin--measurable). For a more modern proof based on Todorcevic's theory of Ramsey spaces, we refer the reader to \cite[Section 5.6]{TodorcevicOrange}

\bigskip

\section{The Banach space \texorpdfstring{$\ell_\infty$}{l infinity}}\label{sec:EllInf}

The main goal of this section is to prove Theorem \ref{thm:OscStabEllInf}. We start with recalling some definitions and basic facts. If $X$ and $Y$ are Banach spaces, we will denote by ${X \choose Y}$ the set of all vector subspaces of $X$ that are linearly isometric to $Y$. In this paper, sequences of real numbers will be seen as functions $\N \to \R$; in particular, the $n$-th entry of a sequence $x$ will be denoted by $x(n)$. The Banach space $\ell_\infty$ is defined as the space of all bounded sequences $x \colon \N \to \R$ endowed with the norm $\|x\|_\infty \coloneq \sup_{n \in \N}|x(n)|$. The Banach space $c_0$ is the subspace of $\ell_\infty$ made of all sequences converging to $0$; it is endowed with the same norm. Finally, $\ell_1$ denotes the Banach space of all sequences $x \colon \N \to \R$ such that $\|x\|_1 \coloneq \sum_{n \in \N}|x(n)| < \infty$, endowed with the norm $\|\cdot\|_1$. Recall that $\ell_\infty$ can be seen as the topological dual of $\ell_1$, via the duality bracket defined by $\langle x, y \rangle \coloneq \sum_{n \in \N} x(n)y(n)$, for $x \in \ell_1$ and $y \in \ell_\infty$. The $\textit{weak-* topology}$ associated to this duality bracket is the coarsest topology on $\ell_\infty$ making all the maps $\langle x, \cdot \rangle \coloneq \ell_\infty \to \R$, $x \in \ell_1$, continuous. The space $\ell_\infty$ can also be endowed by the topology induced by the product topology on $\R^\N$ (simply called \textit{product topology} in what follows); this topology differs from the weak-* topology on the whole space. However it is not hard to see they coincide on the unit sphere $S_{\ell_\infty}$. In this paper we chose to refer to this topology as \textit{the weak-* topology} as this denomination is more common to Banach space theorists, however its presentation as the product topology will be the one used in proofs.

\smallskip

A central element of our proof is the right action $\ell_\infty \curvearrowleft \CS$ by composition. Observe that if $x \in \ell_\infty$ and $p \in \CS$, then, because $p$ is a surjection, we have $\{x(n) \mid n \in \N\} = \{x\circ p(m) \mid m \in \N\}$, so $\|x \circ p\|_\infty = \|x\|_\infty$. Hence, $\CS$ acts by linear isometries on $\ell_\infty$. Also observe that, for fixed $x \in \ell_\infty$, the map $\CS \to \ell_\infty$, $p \mapsto x \circ p$ is continuous when $\ell_\infty$ is endowed with the weak-* topology. Indeed, if $p, q \in \CS$ satisfy $p\restriction_m = q\restriction_m$, then we also have $x\circ p\restriction_m = x \circ q\restriction_m$.

\smallskip

If $X$ is a metric space, $A\subseteq X$, and $\varepsilon > 0$, define the \textit{$\varepsilon$-fattening} of $A$ as the set $(A)_\varepsilon \coloneq \{x \in X \mid (\exists y \in A)(d(x, y)\leqslant \varepsilon)\}$. The main element of our proof is the following proposition.

\begin{prop}\label{prop:MainLInf}
    For every integer $k \geqslant 1$, there exists $x_k \in S_{\ell_\infty}$ and $X_k \in {\ell_\infty \choose \ell_\infty}$ such that $S_{X_k} \subseteq (x_k \circ \CS)_\frac{2}{k}$.
\end{prop}

We first show how to deduce Theorem \ref{thm:OscStabEllInf} from Proposition \ref{prop:MainLInf}.

\begin{proof}[Proof of Theorem \ref{thm:OscStabEllInf}]
    Let $\varepsilon > 0$, $K$ be a compact metric space and $f \colon S_{\ell_\infty} \to K$ be a weak-* Borel, uniformily continuous function. Fix an integer $k \geqslant 1$ with the property that for all $x, y \in S_{\ell}$, if $\|x - y\|_\infty \leqslant 2/k$, then $d(f(x), f(y)) \leqslant \varepsilon / 4$. Since $K$ is compact, we can find an open cover $(U_i)_{i < r}$ of $K$ by finitely many open sets of diameter at most $\varepsilon/2 > 0$; letting $A_i \coloneq U_i \setminus \bigcup_{j < i} U_j$, $(A_i)_{i < r}$ is a partition of $K$ into Borel sets of diameter at most $\varepsilon / 2$. Consider $x_k$ and $X_k$ given by Proposition \ref{prop:MainLInf}. Define a colouring $c \colon \CS \to r$ by $c(p) = i \Leftrightarrow f(x_k \circ p) \in A_i$. By our assumptions, this is a Borel colouring, so by the Carlson--Simpson theorem, we can find $p \in \CS$ such that $c$ is constant on $\CS \circ p$. Denote by $i_0$ the value of this constant. We have $f(x_0 \circ \CS \circ p) \in A_{i_0}$. Recall that $S_{X_k} \subseteq (x_k \circ \CS)_{\frac{2}{k}}$; since $\CS$ acts on $\ell_\infty$ by isometries, we deduce that $S_{X_k \circ p} = S_{X_k} \circ p \subseteq (x_k \circ \CS \circ p)_{\frac{2}{k}}$. By the choice of $k$, we hence have $f(S_{X_k}) \subseteq (f(x_k \circ \CS \circ p))_{\frac{\varepsilon}{4}} \subseteq (A_{i_0})_{\frac{\varepsilon}{4}}$. This concludes the proof, since $\diam((A_{i_0})_{\frac{\varepsilon}{4}}) \leqslant \varepsilon$.
\end{proof}

We now proceed to the proof of Proposition \ref{prop:MainLInf}. For this, we introduce some more notation. If $x \in \ell_\infty$, define the \textit{support} of $x$ as $\supp(x) \coloneq \{n \in \N \mid x(n) \neq 0\}$. If $x_0, x_1, \ldots$ are disjointly supported vectors, denote by $\sum_{i \in \N} x_i$ the vector whose $n$-th entry is $\sum_{i \in \N} x_i(n)$ for every $n \in \N$. Finally, define the \textit{forward shift} on $\ell_\infty$ as the map $S \colon \ell_\infty \to \ell_\infty$ defined by $S(x) = (0, x(0), x(1), \ldots)$. This is obviously a linear isometry. Now fix $k \geqslant 1$; our vector $x_k$ will be defined as follows:
$$x_k = \left(0, \frac{1}{k}, - \frac{1}{k}, \frac{2}{k}, -\frac{2}{k}, \ldots, \frac{k-1}{k}, -\frac{k-1}{k}, 1, -1, \frac{k-1}{k}, -\frac{k-1}{k}, \ldots, \frac{1}{k}, - \frac{1}{k}, 0, 0, 0, \ldots\right).$$
Observe that $\supp(x_k) = \llbracket 1, 4k-2\rrbracket$.

\begin{lem}\label{lem:PropertyOfxk}
    Let $a \in S_{\ell_\infty}$ be such that $a(n) \geqslant 0$ for all $n \in \N$. Let $n_0, n_1, \ldots$ be integers such that for all $i \in \N$, $n_{i+1} - n_i \geqslant 4k-1$. Let $x \coloneq \sum_{i \in \N}a(i)S^{n_i}(x_k)$. Then there exists $p \in \CS$ such that $\|x - x_k\circ p\|_\infty \leqslant 1/2k$.
\end{lem}

\begin{proof}
    For $u \in [-1, 1]$, define $h(u)$ as follows:
    \begin{itemize}
        \item if $u \in \left[-\frac{1}{2k}, \frac{1}{2k}\right]$, then $h(u) = 0$;
        \item if $u \in \left(\frac{2l+1}{2k}, \frac{2l+ 3}{2k}\right]$ for an integer $l \geqslant 0$, then $h(u) = \frac{l+1}{k}$;
        \item if $u \in \left[-\frac{2l+3}{2k}, -\frac{2l+ 1}{2k}\right)$ for an integer $l \geqslant 0$, then $h(u) = -\frac{l+1}{k}$.
    \end{itemize}
    Observe that, for every $u, v \in [0, 1]$, the following holds:
    \begin{enumerate}[(a)]
        \item\label{it:hMult} $h(u)$ is an integer multiple of $1/k$;
        \item\label{it:hApprox} $|h(u) - u| \leqslant 1/2k$;
        \item\label{it:hMinus} $h(-u) = -h(u)$;
        \item\label{it:hRegul} $|u - v| \leqslant 1/k \Rightarrow |h(u) - h(v)| \leqslant 1/k$.
    \end{enumerate}

    \smallskip
    
    Let $y \coloneq h \circ x$. It follows from property (\ref{it:hApprox}) that $\|y - x\|_\infty \leqslant 1/2k$. So to prove our lemma, it is enough to find $p \in \CS$ such that $x_k\circ p = y$. Since $a \in S_{\ell_\infty}$, there exists $i_0 \in \N$ such that $a(i_0) > \frac{2k-1}{2k}$. Letting $N \coloneq n_{i_0} + 2k$, we have $x(N) = a(i_0)x_k(2k) = -a(i_0)$, thus $y(N) = h(-a(i_0)) = -1$. Now observe that the vector $x_k$ satisfies the following two properties:
    \begin{enumerate}[(i)]
        \item\label{it:xkLip} for every $n \in \N$, $||x_k(n+1)| - |x_k(n)|| \leqslant 1/k$;
        \item\label{it:xkSym} given $n \in \N$ and $u > 0$, if $x_k(n) = -u$, then $n \geqslant 1$ and $x_k(n-1) = u$. 
    \end{enumerate}
    It is immediate that the vector $x$ also satisfies properties (\ref{it:xkLip}) and (\ref{it:xkSym}). Using properties (\ref{it:hMinus}) and (\ref{it:hRegul}) of $h$, we deduce that the vector $y$ also satisfies properties (\ref{it:xkLip}) and (\ref{it:xkSym}). Hence the values taken by $y(n)$ for $0 \leqslant n \leqslant N$ should exactly be $0, 1/k, -1/k, \ldots, (k-1)/k, -(k-1)/k, 1, -1$, with this order of first appearance. Observe that those values are exactly $x_k(0), x_k(1), \ldots, x_k(2k)$. Hence, letting, for all $n \leqslant N$, $p(n)$ be the unique $m \leqslant 2k$ such that $x_k(m) = y(n)$, we define a rigid surjection $p \colon (N+1) \to (2k+1)$ such that $y\restriction_{N+1} = x_k\circ p\restriction_{N+1}$.

    \smallskip

    Since $y(N) = -1$ and $y(N+2k-1) = 0$, using the fact that $y$ satisfies properties (\ref{it:xkLip}) and (\ref{it:xkSym}), we deduce that we necessarily have:
    $$(x(N+1), \ldots, x(N+2k-2)) =  \left(\frac{k-1}{k}, -\frac{k-1}{k}, \ldots, \frac{1}{k}, - \frac{1}{k}\right).$$
    Hence, if we extend $p$ by letting, for each $1 \leqslant n \leqslant 2k-2$, $p(N+n) = 2k+n$, the map $p$ becomes a rigid surjection $(N+2k-1) \to (4k-1)$ still satisfying $y\restriction_{N+2k-1} = x_k\circ p\restriction_{N+2k-1}$.

    \smallskip

    Finally, observe that the condition $n_{i+1} - n_i \geqslant 4k-1$ ensures that the vector $x$ contains infinitely many zeros. Hence, we can extend $p$ to a rigid surjection $\N \to \N$ by induction in the following way: given $n \geqslant N+2k-1$,
    \begin{itemize}
        \item if $y(n) \neq 0$, let $p(n)$ be the unique element $m \in \llbracket 1, 2k \rrbracket$ such that $y(n) = x_k(m)$;
        \item if $y(n) = 0$, let $p(n) = \max \{p(0), \ldots, p(n-1)\} +1$.
    \end{itemize}
    We obviously still have $y = x_k \circ p$.
\end{proof}

\begin{proof}[Proof of Proposition \ref{prop:MainLInf}]
    Define $T \colon \ell_\infty \to \ell_\infty$ by $T(a) = \sum_{i \in \N} a(i)S^{4ki}(x_k)$. Observe that the vectors $S^{4ki}(x_k)$ are pairwise disjointly supported, so $T$ is well-defined and is a linear isometry. Thus, letting $X_k = T(\ell_\infty)$, we have $X_k \in {\ell_\infty \choose \ell_\infty}$.

    \smallskip

    Fix $a \in S_{\ell_\infty}$; we will prove the existence of $p \in \CS$ such that $\|T(a) - x_k\circ p\| \leqslant 2/k$. This will be enough to conclude, since all elements of $S_{X_k}$ are of the form $T(a)$, $a \in S_{\ell_\infty}$. For each $i \in \N$, let $n_i \coloneq 4ki$ if $a(i) \geqslant 0$ and $n_i \coloneq 4ki+1$ if $a(i) < 0$. Let $x \coloneq \sum_{i \in \N}|a(i)|S^{n_i}(x_k)$. By Lemma \ref{lem:PropertyOfxk}, we can find $p \in \CS$ such that $\|x - x_k \circ p\| \leqslant 1/2k$. Also note that sets of the form $\supp(|a(i)|S^{n_i}(x_k)) \cup \supp(a(i)S^{4ki}(x_k))$, $i \in \N$, are pairwise disjoint, hence we have:
    $$\|T(a) - x\| = \left\|\sum_{i \in \N}(a(i)S^{4ki}(x_k) - |a(i)|S^{n_i}(x_k))\right\| = \sup_{i \in \N} \|a(i)S^{4ki}(x_k) - |a(i)|S^{n_i}(x_k)\|.$$
    The quantity $\|a(i)S^{4ki}(x_k) - |a(i)|S^{n_i}(x_k)\|$ is nonzero iff $a(i) < 0$, in which case it is equal to $|a(i)| \cdot \|S^{4ki}(x_k) + S^{4ki+1}(x_k)\|$, itself equal to $|a(i)| \cdot \|x_k + S(x_k)\|$, since $S$ is an isometry. Observe that $\|x_k + S(x_k)\| = 1/k$; thus, it follows that $\|T(a) - x\|\leqslant 1/k$. Finally, we have $\|T(a) - x_k \circ p\| \leqslant 1/2k + 1/k \leqslant 2/k$.
\end{proof}

This finishes the proof of Theorem \ref{thm:OscStabEllInf}. The rest of this section is devoted to the proof of the following proposition, which shows that the assumption that $f$ is weak-* Borel cannot be removed from Theorem \ref{thm:OscStabEllInf}.

\begin{prop}\label{prop:diagonal}
    There exists a $1$-Lipschitz map $f \colon S_{\ell_\infty} \to [0, 1]$ such that for all $X \in {\ell_\infty \choose \ell_\infty}$, $f(S_X) = [0, 1]$.
\end{prop}

Given $A \subseteq \N$, we abusively let $\ell_\infty(A) \coloneq \{x  \in \ell_\infty \mid \supp(x) \subseteq A\}$; for infinite $A$, this is an element of $\ell_\infty \choose \ell_\infty$. To prove Proposition \ref{prop:diagonal} we will need the two following following facts, well-known to Banach space theorists (see \cite[Proposition 2.5.2]{AlbiacKalton} and \cite[Theorem 2.5.4]{AlbiacKalton}, respectively).

\begin{prop}\label{prop:InjLInf}
    Let $X$ be a Banach space. If $Y \in {X \choose \ell_\infty}$, then there exists a linear continuous projection $P \colon X \to Y$ having operator norm $1$ (i.e. such that $\|P(x)\| \leqslant \|x\|$ for every $x \in X$).
\end{prop}

\begin{prop}\label{prop:OperatorsOnLIinf}
    Let $T \colon \ell_\infty \to \ell_\infty$ be a continuous linear map. Suppose that $T$ vanishes on $c_0$. Then $T$ vanishes on $\ell_\infty(A)$ for some infinite $A \subseteq \N$.
\end{prop}

The proof of Proposition \ref{prop:diagonal} relies on a diagonal argument; the use of such an argument to build counterexamples for non-definable colourings in infinite-dimensional Ramsey theory is classical. However, due to the specificities of the present situation, we will need two simple preliminary lemmas.

\begin{lem}\label{lem:CofinalFamilyCopies}
    There exists a family $\mathcal{F} \subseteq {\ell_\infty \choose \ell_\infty}$, having size continuum, such that for every $X \in {\ell_\infty \choose \ell_\infty}$, there exists $Y \in \mathcal{F}$ with $Y \subseteq X$.
\end{lem}

\begin{proof}
    Let $\mathcal{I}$ be the set of all linear isometries $\ell_\infty \to \ell_\infty$. On $\mathcal{I}$ we define an equivalence relation $\sim$ by $S \sim T \Leftrightarrow S\restriction_{c_0} = T\restriction_{c_0}$. Note that a linear isometry $c_0 \to \ell_\infty$ is entirely determined by the images of vectors of the canonical basis of $c_0$; hence there are continuum-many isometries $c_0 \to \ell_\infty$. Thus the equivalence relation $\sim$ has continuum-many classes. Let $\mathcal{T} \subseteq \mathcal{I}$ with $|\mathcal{T}| = \mathfrak{c}$ be a set intersecting each class of $\sim$. We let $$\mathcal{F} \coloneq \{T(\ell_\infty(A)) \mid T \in \mathcal{T}, A \subseteq \N \text{ infinite}\}.$$
    Let us show that $\mathcal{F}$ is as wanted. First, we clearly have $|\mathcal{F}| \leqslant \mathfrak{c}$. Now, given an arbitrary $X \in {\ell_\infty \choose \ell_\infty}$, one can find $S \in \mathcal{I}$ such that $X = S(\ell_\infty)$. Let $T \in \mathcal{T}$ be such that $T \sim S$. By Proposition \ref{prop:OperatorsOnLIinf}, there should exist an infinite $A \subseteq \N$ such that $T\restriction_{\ell_\infty(A)} = S\restriction_{\ell_\infty(A)}$. Hence, letting $Y \coloneq S(\ell_\infty(A)) = T(\ell_\infty(A))$, we have $Y \in \mathcal{F}$ and $Y \subseteq X$.
    \end{proof}

    \begin{lem}\label{lem:SeparatedFamilies}
        Let $X \in {\ell_\infty \choose \ell_\infty}$ and $A \subseteq S_{\ell_\infty}$
        with $|A| < \mathfrak{c}$. Then there exists $x \in S_X$ with $d(x, A) \geqslant 1$.
    \end{lem}

    \begin{proof}
        We first prove the result in the special case when $X = \ell_\infty$. Define $s \colon [-1, 1] \to \{-1, 1\}$ by $s(x) = 1 \Leftrightarrow x \geqslant 0$. Then $|\{s\circ a \mid a \in A\}| < \mathfrak{c}$, so there exists $x \in \{-1, 1\}^\N$ such that $x \neq s\circ a$ for all $a \in A$. It easily follows that $d(x, A) \geqslant 1$. 

        \smallskip

        In the general case, use Proposition \ref{prop:InjLInf} to find a continuous linear projection $P \colon \ell_\infty \to X$ of operator norm $1$. Let $B \coloneq P(A)$, a subset of $S_X$ with $|B| < \mathfrak{c}$. Applying the special case to $B$ in the subspace $X$, we can find $x \in S_X$ such that $d(x, B) \geqslant 1$. Now, if $a \in A$, we have $\|x - a\| \geqslant \|P(x-a)\| = \|x - P(a)\| \geqslant 1$, so $d(x, A) \geqslant 1$.
    \end{proof}

    \begin{proof}[Proof of Proposition \ref{prop:diagonal}]
        Let $\mathcal{F} \subseteq {\ell_\infty \choose \ell_\infty}$ be the family given by Lemma \ref{lem:CofinalFamilyCopies}. Enumerate $\mathcal{F} \eqcolon \{X_\alpha \mid \alpha < \mathfrak{c}\}$. We recursively build two families $(a_\alpha)_{\alpha < \mathfrak{c}}$ and $(b_\alpha)_{\alpha < \mathfrak{c}}$ as follows: given $\alpha < \mathfrak{c}$, using Lemma \ref{lem:SeparatedFamilies}, we can choose $a_\alpha, b_\alpha \in S_{X_\alpha}$ such that $d(a_\alpha, \{b_\beta \mid \beta < \alpha\}) \geqslant 1$ and  $d(b_\alpha, \{a_\beta \mid \beta < \alpha\}) \geqslant 1$. Letting $A \coloneq \{a_\alpha \mid \alpha < \mathfrak{c}\}$ and $B \coloneq \{b_\alpha \mid \alpha < \mathfrak{c}\}$, we have that $d(A, B) \geqslant 1$; moreover, by the choice of $\mathcal{F}$, the sets $A$ and $B$ both intersect the unit spheres of all elements of ${\ell_\infty \choose \ell_\infty}$. We now let, for $x \in S_{\ell_\infty}$, $f(x) \coloneq d(x, A)$; the map $f$ satisfies the desired properties.
    \end{proof}

\bigskip

\section{The Urysohn sphere}\label{sec:Urysohn}

A metric space $X$ is said to be \textit{ultrahomogeneous} if every isometry $T \colon A \to X$, where $A \subseteq X$ is finite, extends to an onto isometry $X \to X$. The \textit{Urysohn sphere} $\mathcal{S}$ is the unique ultrahomogeneous complete metric space of diameter $1$ containing isometric copies of all finite metric spaces of diameter at most $1$.

\smallskip

Denote by $\mathfrak{S}_1$ the class of all separable metric spaces with diameter at most $1$. Say that a metric space $X$ is \textit{$\mathfrak{S}_1$-universal} if $X \in \mathfrak{S}_1$ and if every $Y \in \mathfrak{S}_1$ isometrically embeds into $X$.  It is a classical result that $\mathcal{S}$ is $\mathfrak{S}_1$-universal. Say that a metric space is \textit{oscillation stable} if it satisfies the conclusion of Theorem \ref{thm:NVTS}. Observe that if two metric spaces $X$ and $Y$ are pairwise isometrically bi-embeddable and if $X$ is oscillation stable, then so is $Y$. Therefore, to prove Theorem \ref{thm:NVTS}, we just need to find \textit{one} $\mathfrak{S}_1$-universal space $\mathcal{U}$ that is oscillation stable. Our space $\mathcal{U}$ will be defined as a space of sequences, in order to be able to define an action of $\CS$ on it. For technical reasons, it will not be a metric space but only a pseudometric space. Recall that a \textit{pseudometric space} is a space $(X, d)$ where the map $d \colon X^2 \to [0, \infty)$ (called a \textit{pseudometric}) satisfyies all axioms of a metric except that one can have $d(x, y) = 0$ even for $x \neq y$. Given a pseudometric space $(X, d)$, one can quotient $X$ by the equivalence relation $\sim$ defined by $x \sim y \Leftrightarrow d(x, y) = 0$; $d$ gives rise to a metric on the quotient. In this paper, we will often identify a pseudometric space and its canonical quotient; it will never be an issue.

\smallskip

Let $\mathcal{U} \coloneq \{x \in [0, 1]^\omega \mid \liminf x = 0\}$. For $x, y \in \mathcal{U}$ and $n \in \omega$, define:
\begin{itemize}
    \item $m(x, y, n) \coloneq \sup_{k \leqslant n} |x(k) - y(k)|$;
    \item $M(x, y, n) \coloneq \inf_{k \leqslant n} (x(k) + y(k))$.
\end{itemize}
Also let $m(x, y, \omega) \coloneq \sup_{n \in \omega} m(x, y, n)$ and $M(x, y, \omega) \coloneq \inf_{n \in \omega} M(x, y, n)$. 

\begin{prop}\label{prop:PropertiesMinMax}
Let $x, y, z \in \mathcal{U}$ and $n \in \omega \cup \{\omega\}$.
\begin{enumerate}[(1)]
\item The map $m(x, y, \cdot) \colon \omega \cup \{\omega\} \to [0, 1]$ is nondecreasing.
\item The map $M(x, y, \cdot) \colon \omega \cup \{\omega\} \to [0, 2]$ is nonincreasing.
\item $m(x, y, 0) \leqslant M(x, y, 0)$.
\item $M(x, y, \omega) \leqslant m(x, y, \omega)$.
\item $m(x, z, n) \leqslant m(x, y, n) + m(y, z, n)$.
\item $M(x, z, n) \leqslant M(x, y, n) + m(y, z, n)$.
\end{enumerate}
\end{prop}

\begin{proof}
(1), (2) and (3) are immediate. (5) is obtained from the inequality $|x(k) - z(k)| \leqslant |x(k) - y(k)| + |y(k) - z(k)|$ by passing to the supremum on both sides. (6) is obtained from the inequality $x(k) + z(k) \leqslant (x(k) + y(k)) + |y(k) - z(k)|$ by first passing $|y(k) - z(k)|$ to the supremum, and then passing both sides to the infimum.

\smallskip

We now prove (4). Fix $\varepsilon > 0$. Since $\liminf x = 0$, we can find $k \in \omega$ such that $x(k) \leqslant \varepsilon$. We then have $M(x, y, \omega) \leqslant x(k) + y(k) \leqslant y(k) + \varepsilon$, and $m(x, y, \omega) \geqslant |y(k) - x(k)| \geqslant y(k) - \varepsilon$. Thus, $M(x, y, \omega) \leqslant m(x, y, \omega) + 2\varepsilon$, and making $\varepsilon$ tend to $0$, we get the desired inequality.
\end{proof}

\begin{prop}\label{prop:AtLeastOneConstant}
    Let $x, y \in \mathcal{U}$ and $n \in \omega$. Then at least one of the following conditions is satisfied:
    \begin{enumerate}[(1)]
    \item $m(x, y, n) = m(x, y, n+1)$;
    \item $M(x, y, n) = M(x, y, n+1)$;
    \item $m(x, y, n+1) \leqslant M(x, y, n+1)$.
    \end{enumerate}
\end{prop}

\begin{proof}
    Suppose neither (1) nor (2) are satisfied. Condition (1) not being satisfied implies that $m(x, y, n+1) = |x(n+1) - y(n+1)|$. Condition (2) not being satisfied implies that $M(x, y, n+1) = x(n+1) + y(n+1)$. Since $|x(n+1) - y(n+1)| \leqslant x(n+1) + y(n+1)$, condition (3) follows.
\end{proof}

We now define a map $d \colon \mathcal{U}^2 \to [0, 1]$ as follows. Consider $x, y \in \mathcal{U}$. Let $n_{x, y} \coloneq \min\{n \in \omega \cup \{\omega\} \mid M(x, y, n) \leqslant m(x, y, n)\}$, which exists by condition (4) in Proposition \ref{prop:PropertiesMinMax}. Define $d(x, y)$ as follows.
\begin{itemize}
    \item If $M(x, y, n_{x, y}) = m(x, y, n_{x, y})$, then $d(x, y)$ is this common value.
    \item Otherwise, by condition (3) in Proposition \ref{prop:PropertiesMinMax}, we can write $n_{x, y} = k + 1$, with $k \in \omega$. By Proposition \ref{prop:AtLeastOneConstant}, either $m(x, y, k) = m(x, y, k+1)$, or $M(x, y, k) = M(x, y, k+1)$ (and both cases cannot occur simultaneously, otherwise it would violate the minimality of $n_{x, y}$). In the first case, let $d(x, y) \coloneq m(x, y, k)$, and in the second case, let $d(x, y) \coloneq M(x, y, k)$.
\end{itemize}
Observe that, in any case, we have $d(x, y) \leqslant m(x, y, n_{x, y})$. This inequality will often be used in our later arguments.

\begin{prop}
    The map $d$ is a pseudometric on $\mathcal{U}$.
\end{prop}

\begin{proof}
    For all $x, y \in \mathcal{U}$ and $n \in \omega \cup \{\omega\}$, we have $m(x, y, n) = m(y, x, n)$ and $M(x, y, n) = M(y, x, n)$, from which it follows that $d(x, y) = d(y, x)$. We also have, for all $x \in \mathcal{U}$, $d(x, x) \leqslant m(x, x, n_{x, x}) = 0$, thus $d(x, x) = 0$. Now it remains the verify that the triangle inequality holds.

    \smallskip

    We start with extending the maps $M$ and $m$ to real values of $n$. Although this step is not formally required, it makes our argument easier to follow. For this, given $x, y \in \mathcal{U}$, we extend the maps $m(x, y, \cdot)$ and $M(x, y, \cdot)$ to $\R_+ \cup \{\omega\}$ so that they are affine on intervals of the form $[n, n+1]$, $n \in \omega$ (here, $\R_+$ denotes the interval $[0, +\infty[$). Note that those extensions are still respectively nondecreasing and nonincreasing, and inequalities (5) and (6) from Proposition \ref{prop:PropertiesMinMax} can be immediately extended to real values, that is, for all $x, y, z \in \mathcal{U}$ and $t \in \R_+ \cup\{\omega\}$, we have
    \begin{equation}
        m(x, z, t) \leqslant m(x, y, t) + m(y, z, t)\label{Ineq:01}
    \end{equation}
    and
    \begin{equation}
        M(x, z, t) \leqslant M(x, y, t) + m(y, z, t).\label{Ineq:02}
    \end{equation}

    \begin{claim}\label{claim:PseudoMetric}
    For all $x, y \in \mathcal{U}$, there exists $t_{x, y} \in \R_+ \cup \{\omega\}$ such that $m(x, y, t_{x, y}) = M(x, y, t_{x, y}) = d(x, y)$.
    \end{claim}
    \begin{proof}
        If $M(x, y, n_{x, y}) = m(x, y, n_{x, y})$, then we can just take $t_{x, y} \coloneq n_{x, y}$. Otherwise, $n_{x, y} = k+1$ for some $k \in \omega$, and one of the following two conditions hold: either
        $$m(x, y, k) < M(x, y, k) = d(x, y) = M(x, y, k+1) < m(x, y, k+1)$$
        or
        $$M(x, y, k) > m(x, y, k) = d(x, y) = m(x, y, k+1) > M(x, y, k+1).$$
        Both cases can be handled similarly, so suppose we are in the first case. In this case, the map $m(x, y, \cdot)$ is affine on the interval $[k, k+1]$ and the map $M(x, y, \cdot)$ is constant with value $d(x, y)$ on the same interval, so a suitable $t_{x, y}$ can be found in this interval by the intermediate value theorem.
    \end{proof}
    Now, fix $x, y, z \in \mathcal{U}$ and prove that $d(x, z) \leqslant d(x, y) + d(y, z)$. Choose $t_{x, y}, t_{x, z}, t_{y, z} \in \R_+ \cup \{\omega\}$ as given by Claim \ref{claim:PseudoMetric}. By symmetry, we can assume that $t_{x, y} \leqslant t_{y, z}$. We distinguish between three cases.
    \begin{itemize}
        \item If $t_{x, z} \leqslant t_{x, y} \leqslant t_{y, z}$, then using Inequality \ref{Ineq:01} and the monotonicity of the maps $m(x, y, \cdot)$ and $m(y, z, \cdot)$, we get:
        \begin{align*}
           d(x, z) &= m(x, z, t_{x, z}) \\
           &\leqslant m(x, y, t_{x, z}) + m(y, z, t_{x, z})\\
           &\leqslant m(x, y, t_{x, y}) + m(y, z, t_{y, z})\\
           &= d(x, y) + d(y, z).
        \end{align*}
        \item If $t_{x, y} \leqslant t_{x, z} \leqslant t_{y, z}$, then using Inequality \ref{Ineq:02} and the monotonicity of the maps $M(x, y, \cdot)$ and $m(y, z, \cdot)$, we get:
        \begin{align*}
           d(x, z) &= M(x, z, t_{x, z}) \\
           &\leqslant M(x, y, t_{x, z}) + m(y, z, t_{x, z})\\
           &\leqslant M(x, y, t_{x, y}) + m(y, z, t_{y, z})\\
           &= d(x, y) + d(y, z).
        \end{align*}
        \item If $t_{x, y} \leqslant t_{y, z} \leqslant t_{x, z}$, then using Inequality \ref{Ineq:02} and the monotonicity of the maps $M(x, z, \cdot)$ and $M(x, y, \cdot)$, we get:
        \begin{align*}
           d(x, z) &= M(x, z, t_{x, z}) \\
           &\leqslant M(x, z, t_{y, z}) \\
           &\leqslant M(x, y, t_{y, z}) + m(y, z, t_{y, z})\\
           &\leqslant M(x, y, t_{x, y}) + m(y, z, t_{y, z})\\
           &= d(x, y) + d(y, z).
        \end{align*}
    \end{itemize}
\end{proof}

\begin{prop}
    $\mathcal{U} \in \mathfrak{S}_1$.
\end{prop}

\begin{proof}
    For all $x, y \in \mathcal{U}$, we have $d(x, y) \leqslant m(x, y, n_{x, y}) \leqslant 1$, thus $\diam(\mathcal{U}) \leqslant 1$. We now need to prove that $\mathcal{U}$ is separable. For this, let $A \coloneq \{x \in \mathcal{U} \mid (\forall n \in \omega)(x(n) \in \mathbb{Q}) \wedge (\exists n \in \omega)(\forall m \geqslant n)(x(m) = 0)\}$ and prove that $A$ is dense in $\mathcal{U}$. Fix $x \in \mathcal{U}$ and $\varepsilon > 0$. Let $n \in \omega$ be such that $x(n) \leqslant \varepsilon$. Let $y \in A$ be such that for all $k < n$, $|x(k) - y(k)| \leqslant \varepsilon$, and for all $k \geqslant n$, $y(k) = 0$. Using that $y(n) = 0$, we get that $M(x, y, n) \leqslant x(n) \leqslant m(x, y, n)$, thus $n_{x, y} \leqslant n$. We deduce that $d(x, y) \leqslant m(x, y, n) \leqslant \varepsilon$.
\end{proof}

Denote by $\CS$ the monoid of rigid surjections $\omega \to \omega$. From now on, we consider the right action $\mathcal{U} \curvearrowleft \CS$ by composition.

\begin{prop}
    The monoid $\CS$ acts by isometries on $\mathcal{U}$.
\end{prop}

\begin{proof}
    Let $p \in \CS$ and $x, y \in \mathcal{U}$. Define an increasing map $h \colon \omega \cup \{\omega\} \to \omega \cup\{\omega\}$ by $h(k) = \min\{n \in \omega \mid p(n) = k\}$ for $k \in \omega$, and $h(\omega) = \omega$. Observe that for all $n, k \in \omega$ such that  $h(k) \leqslant n < h(k+1)$, or for $n = k = \omega$, we have
    \begin{equation}
    m(x \circ p, y\circ p, n) = m(x, y, k)\text{ and }M(x \circ p, y\circ p, n) = M(x, y, k).\label{Eq:025}
    \end{equation}
    Thus $n_{x \circ p, y \circ p} = h(n_{x, y})$, $m(x \circ p, y \circ p, n_{x \circ p, y\circ p}) = m(x, y, n_{n, y})$, and $M(x \circ p, y \circ p, n_{x \circ p, y\circ p}) = M(x, y, n_{n, y})$. We now treat separately the three cases in the definition of $d(x, y)$.

    \begin{description}
        \item[1\textsuperscript{st} case:] $d(x, y) = M(x, y, n_{x, y}) = m(x, y, n_{x, y})$. In this case, $M(x \circ p, y \circ p, n_{x \circ p, y\circ p})  = m(x \circ p, y \circ p, n_{x \circ p, y\circ p})$, so $d(x \circ p, y\circ p) = m(x \circ p, y \circ p, n_{x \circ p, y\circ p}) = m(x, y, n_{n, y}) = d(x, y)$.

        \item[2\textsuperscript{nd} case:] $M(x, y, n_{x, y}) < m(x, y, n_{x, y})$, $n_{x, y} = k+1$ for some $k \in \omega$, and $d(x, y) = m(x, y, k) = m(x, y, k+1)$. Then we have $M(x \circ p, y \circ p, n_{x\circ p, y\circ p}) < m(x \circ p, y \circ p, n_{x\circ p, y\circ p})$. Moreover, writing $n_{x \circ p, y \circ p} = l+1$ for some $l \in \omega$, we have $h(k) \leqslant l < h(k+1)$, so by Condition \ref{Eq:025}, we have $m(x \circ p, y\circ p, l) = m(x, y, k)$. Thus, $m(x \circ p, y\circ p, l) = m(x, y, k+1) = m(x \circ p, y\circ p, l+1)$, so $d(x \circ p, y\circ p) = m(x \circ p, y\circ p, l)$. Hence,
        $$d(x \circ p, y\circ p) = m(x \circ p, y\circ p, l) = m(x, y, k) = d(x, y).$$

        \item[3\textsuperscript{rd} case:] $M(x, y, n_{x, y}) < m(x, y, n_{x, y})$, $n_{x, y} = k+1$ for some $k \in \omega$, and $d(x, y) = M(x, y, k) = M(x, y, k+1)$. This case is treated similarly as the second case.
    \end{description}%
\end{proof}

For $r \geqslant 1$, define $w_r \in \mathcal{U}$ as follows: for $n \leqslant r$, $w_r(n) \coloneq \frac{r-n}{r}$, and for $n > r$, $w_r(n) \coloneq 0$. Let $\mathcal{U}_r \coloneq (w_r \circ \CS)_{\frac{1}{2r}}$. We see $\mathcal{U}_r$ as a metric subspace of $\mathcal{U}$.

\begin{prop}
    For every $r \geqslant 1$, the metric space $\mathcal{U}_r$ is $\mathfrak{S}_1$-universal.
\end{prop}

\begin{proof}
    Fix $r \geqslant 1$. We already know that $\mathcal{U}_r$ is in $\mathfrak{S}_1$, since it is a subset of $\mathcal{U}$. We now fix $X \in \mathfrak{S}_1$ and show that $X$ isometrically embeds into $\mathcal{U}_r$.

    \smallskip
    
    First extend $X$ to a metric space $X_*:= X \sqcup \{*\}$, where for all $x \in X$, $d(x, *) = 1$. Recall that $X_*$ isometrically embeds into the Banach space $\mathcal{C}_b(X_*)$ of real-valued bounded continuous functions on $X_*$, endowed with the supremum norm: map $x \in X_*$ to the function $y \mapsto d(x, y)$. Identifing $X_*$ with its image by this embedding, we can assume that $X_*$ is a subset of a normed vector space. Denote by $Y$ the convex hull of $X_*$; the metric space $Y$ is itself in $\mathfrak{S}_1$.

    \smallskip

    Now fix a dense sequence $(y_n)_{n \in \omega}$ in $Y$, with $y_0 = *$. Using convexity, adding entries in the sequence if necessary, we can assume that for all $n \in \omega$, $d(y_n, y_{n +1}) \leqslant \frac{1}{r}$. We now define a map $f \colon X \to [0, 1]^\omega$ as follows: for $x \in X$ and $n \in \omega$, let $f(x)(n) \coloneq d(x, y_n)$. Since $(x_n)$ is dense in $Y$, it follows that $f$ actually takes values in $\mathcal{U}$. We now show that $f \colon X \to \mathcal{U}$ is an isometric embedding. Fix $x, x' \in X$. For every $n \in \omega$, we have
    \begin{multline}
    |f(x)(n) - f(x')(n)| = |d(x, y_n) - d(x', y_n)| \\ \leqslant d(x, x') \\ \leqslant d(x, y_n) + d(x', y_n) = f(x)(n) + f(x')(n)
    \end{multline}
    so passing to the supremum on the left hand side and to the infimum on the right hand side, we get that $m(f(x), f(x'), \omega) \leqslant d(x, x') \leqslant M(f(x), f(x'), \omega)$. From property (4) in Proposition \ref{prop:PropertiesMinMax}, it follows that those inequalities are equalities. Thus, by the definition of the metric on $\mathcal{U}$, we have $d(f(x), f(x')) = d(x, x')$.

    \smallskip

    It now remains to prove that $f(X) \subseteq \mathcal{U}_r$. For this, we fix $x \in X$, and we look for $p \in \CS$ such that $d(w_r \circ p, f(x)) \leqslant \frac{1}{2r}$. First observe that since $y_0 = *$ and for all $n \in \omega$, $d(x_n, x_{n+1}) \leqslant \frac{1}{r}$, we have
    \begin{equation}
        f(x)(0) = 1\text{ and for all }n \in \omega\text{, }|f(x)(n) - f(x)(n+1)| \leqslant \frac{1}{r}.\label{Eq:03}
    \end{equation}
    Fix $n_0 \in \omega$ such that $f(x)(n_0) \leqslant \frac{1}{2r}$. From Conditions \ref{Eq:03} it follows that $n_0 \geqslant r$. Now define $p \colon \omega \to \omega$ as follows:
    \begin{itemize}
        \item for $n \leqslant n_0$, let $p(n)$ be the unique integer $k \leqslant r$ such that $$f(x)(n) \in \left(\frac{2(r-k)-1}{2r}, \frac{2(r-k)+1}{2r}\right];$$
        \item for $n > n_0$, let $p(n) \coloneq n - n_0 + r$.
    \end{itemize}
Observe that the equality $p(n) = n - n_0 + r$ is still satisfied for $n = n_0$. It follows from Conditions \ref{Eq:03} that $p(0) = 0$ and for every $n \in \omega$, $|p(n+1) - p(n)| \leqslant 1$. Moreover, $p$ is unbounded, so $p$ is a rigid surjection. Now observe that for all $n \leqslant n_0$, we have $|f(x)(n) - w_r(p(n))| \leqslant \frac{1}{2r}$, thus $m(f(x), w_r\circ p, n_0) \leqslant \frac{1}{2r}$. Since $w_r(p(n_0)) = 0$, we have $M(f(x), w_r\circ p, n_0) \leqslant f(x)(n_0) \leqslant m(f(x), w_r\circ p, n_0)$, thus $n_{f(x), w_r\circ p} \leqslant n_0$, hence $d(f(x), w_r\circ p) \leqslant m(f(x), w_r\circ p, n_0) \leqslant \frac{1}{2r}$.
\end{proof}

\begin{prop}\label{prop:UrysohnBorel}
For every $r \geqslant 1$, the map $\CS \to \mathcal{U}$, $p \mapsto w_r \circ p$ is continous.
\end{prop}

\begin{proof}
    Let $p \in \CS$. Fix $n \in \omega$ such that $p(n) = r$, and let $V \coloneq \{q \in \CS \mid (\forall k \leqslant n)(q(k) = p(k))\}$, a neighborhood on $p$ in $\CS$. We prove that for all $q \in V$, $d(w_r \circ p, w_r \circ q) = 0$. Fix $q \in V$. For all $k \leqslant n$, we have $w_r(p(k)) = w_r(q(k))$, thus $m(w_r\circ p, w_r\circ q, n) = 0$. Moreover, $w_r(p(n)) = w_r(q(n)) = 0$, thus $M(w_r\circ p, w_r\circ q, n) = 0$. It follows that $d(w_r \circ p, w_r \circ q) = 0$.
\end{proof}

\begin{proof}[Proof of Theorem \ref{thm:NVTS}]
Let $K$ be a compact metric space, $f \colon \mathcal{S} \to K$ be uniformily continuous, and $\varepsilon > 0$. Since $\mathcal{S}$ is $\mathfrak{S}_1$-universal and $\mathcal{U} \in \mathfrak{S}_1$, one can assume that $\mathcal{S}$ is actually a superset of $\mathcal{U}$. Let $r \geqslant 1$ be such that for all $x, y \in \mathcal{S}$, $d(x, y) \leqslant \frac{1}{2r} \implies d(f(x), f(y)) \leqslant \frac{\varepsilon}{3}$. 
Let $(B_j)_{j < l}$ be a partition of $K$ in finitely many Borel pieces of diameter at most $\frac{\varepsilon}{3}$. Define $c \colon \CS \to l$ as follows: for $p \in \CS$, $c(p)$ is such that $f(w_r \circ p) \in B_{c(p)}$. It follows from Proposition \ref{prop:UrysohnBorel} that $c$ is a Borel map. Hence, by Carlson--Simpson's theorem, we can find $p \in \CS$ such that $c$ is constant on $CS \circ p$. Let $j$ denote this constant. We thus have $f(w_r\circ \CS \circ p) \subseteq B_j$. Since right-composition by $p$ is isometric, we have $\mathcal{U}_r \circ p \subseteq (w_r \circ \CS \circ p)_\frac{1}{2r}$. By the choice of $r$, we have $f(\mathcal{U}_r \circ p) \subseteq f((w_r \circ \CS \circ p)_\frac{1}{2r}) \subseteq (B_j)_{\frac{\varepsilon}{3}}$. It follows that $\osc(f\restriction_{\mathcal{U}_r \circ p}) \leqslant \diam((B_j)_{\frac{\varepsilon}{3}}) \leqslant \varepsilon$. Since $\mathcal{U}_r \circ p$ is isometric to $\mathcal{U}_r$, it is $\mathfrak{S}_1$-universal, so we can find an isometric copy $\mathcal{S}'$ of $\mathcal{S}$ in $\mathcal{U}_r \circ p$. This copy satisfies the desired condition.
\end{proof}

\bigskip\bigskip

\bibliographystyle{plain}
\bibliography{main}

\end{document}